\documentclass[11pt]{amsart}
\usepackage{amssymb}
\usepackage{amsfonts}
\usepackage{amsmath}
\usepackage{amsthm}
\usepackage{graphicx}
\usepackage{enumerate}

\usepackage{hyperref}

\usepackage[style=ieee,backend=bibtex]{biblatex}
\bibliography{References.bib}
\newtheorem{corollary}{Corollary}[section]
\newtheorem{theorem}[corollary]{Theorem}

\newtheorem{definition}[corollary]{Definition}
\newtheorem{remark}[corollary]{Remark}
\newtheorem{example}[corollary]{Example}

\newcommand{\ff}{\ensuremath{\mathbb{F}}}

\mathchardef\mhyphen="2D

\newcommand{\design}[2][t]{#1 \,\mhyphen\!\left(#2\right)}

\def\Z{\mathbb{Z}}

\begin{document}

\title{Mosaics of Combinatorial Designs }
\author[Oliver W.~Gnilke]{Oliver Wilhelm Gnilke}
\email{oliver.gnilke@aalto.fi or @gmail.com}
\address{Department of Mathematics and Systems Analysis, 
	Aalto University, P.O. Box 11100
	\hbox{FI-00076 Aalto,} Finland}

\author{Marcus Greferath}
\email{marcus.greferath@aalto.fi or @ucd.ie}
\address{Department of Mathematics and Systems Analysis, 
	Aalto University, P.O. Box 11100
	\hbox{FI-00076 Aalto,} Finland}

\author[Mario O.~Pav\v{c}evi\'{c}]{Mario Osvin Pav\v{c}evi\'{c}}
\email{mario.pavcevic@fer.hr}
\address{University of Zagreb, Faculty of electrical engineering
and computing, Department of applied mathematics, Unska~3,
\hbox{HR-10000 Zagreb,} Croatia} 

\thanks{The authors gratefully acknowledge support provided by the European Science Foundation in the context of COST Action IC1104: \emph{Random Network Coding and Designs over GF(q)}.}
\thanks{This work has been supported in part by the Croatian Science Foundation under the project 1637}

\date{March, 2015}

\begin{abstract}
Looking at incidence matrices of $t$-$(v,k,\lambda)$ designs as $v \times b$ matrices with $2$ possible entries, each of which indicates incidences of a $t$-design, we introduce the notion of a $c$-mosaic of designs, having the same number of points and blocks, as a matrix with $c$ different entries, such that each entry defines incidences of a design. In fact, a $v \times b$ matrix is decomposed in $c$ incidence matrices of designs, each denoted by a different colour, hence this decomposition might be seen as a tiling of a matrix with incidence matrices of designs as well. These mosaics have applications in experiment design when considering a simultaneous run of several different experiments. We have constructed infinite series of examples of mosaics and state some probably non-trivial open problems. Furthermore we extend our definition to the case of $q$-analogues of designs in a meaningful way. 
\end{abstract}
\keywords{$t$-design, $c$-mosaic, affine plane, resolvable design}
\subjclass{05B30, 05B05}

\maketitle

\section{Introduction}

A $t$-design with parameters $(v,k,\lambda)$ is a collection ${\mathcal B}$ of $k$-element subsets (blocks) of a $v$-element set $X$ (of points), such that every $t$-element subset of $X$ is contained in exactly $\lambda$ blocks. In such a case we speak sometimes of a $\design{v,k,\lambda}$ design.

It is known that $t$, $v$, $k$ and $\lambda$ must satisfy a number of more or less complicated (necessary) divisibility conditions whereas for $t\geq 2$ the existence of a $t$-design with parameters $(v,k,\lambda)$ is only known for particular cases which are not described by any general set of sufficient conditions.

A particularly convenient way to represent a design ${\mathcal B}$ is a binary $(v \times b)$-matrix, where $b$ is the number of blocks of the design and can be calculated as $b=\lambda \cdot {v \choose t} / {k \choose t}$. This matrix, also called the incidence matrix of ${\mathcal B}$, is labeled by the points in $X$ and the blocks of ${\mathcal B}$, and its entries give the value of the incidence, namely $1$ for incident and $0$ for non-incident. Each point of a $t$-design is incident with the same number of blocks, usually denoted by $r=\lambda \cdot {v-1 \choose t-1} / {k-1 \choose t-1}$.

The complement $\overline{\mathcal B}$ of a design $\mathcal B$ has the same point set $X$, while its blocks are complements of the blocks of {$\mathcal B$}. The complement of a $t$-design is a $t$-design as well. In particular, the complement of a $\design[2]{v,k,\lambda}$ design is a $\design[2]{v,v-k,b-2r+\lambda}$ design. Therefore, one can say that the entries equal to 1 in an incidence matrix indicate incidences of a design, while the 0's  indicate incidences of its complement.  

In the paper \cite{greftherk}, dealing primarily with $1$-designs, the notion of a coloured constant-composition design was thought of as a generalisation of this binary understanding of incidence in that the incidence matrix may be multi-valued (coloured) such that the set of subsets defined by every colour forms a design with its individual parameter set. A simple example for this was given in \cite[example 1.3]{greftherk} and looks as follows.

\begin{example} \label{ex1}
The matrix  $$\left[\begin{array}{ccccccc}
2 & 3 & 3 & 1 & 3 & 1 & 1\\
1 & 2 & 3 & 3 & 1 & 3 & 1\\
1 & 1 & 2 & 3 & 3 & 1 & 3\\
3 & 1 & 1 & 2 & 3 & 3 & 1\\
1 & 3 & 1 & 1 & 2 & 3 & 3\\
3 & 1 & 3 & 1 & 1 & 2 & 3\\
3 & 3 & 1 & 3 & 1 & 1 & 2
\end{array}\right]$$
is the incidence matrix of a $3$-coloured constant composition design in the sense of \cite{greftherk}. 
\end{example}

As a motivation for the work on the paper at hand, we observed that colours $1$ and $3$ each define a $2$-design with parameters $(7,3,1)$ whereas colour $2$ gives rise to the trivial $2$-design with parameters $(7,1,0)$. 

Indeed, all the $v \times b$ incidences between the points of $X$ and 
$b$ copies of $X$ itself can be partitioned in a mathematical constellation that might be best described by the expression:
$$\design[2]{7,3,1} \;\oplus\; \design[2]{7,3,1} \;\oplus\; \design[2]{7,1,0}.$$
It may get even clearer if we remark that the all-one matrix $J$ of dimension $v \times b$ can be written as a sum of the corresponding incidence matrices of mentioned $2$-designs.

As a matter of fact, we are able to immediately generalise this basic idea to any prime power $v\equiv 3\mod{4}$ using quadratic residues and non-residues in ${\rm GF}(v)$, forming developments of the Paley
difference sets in the additive groups of $GF(v)$, so that we clearly and constructively obtain $3$-valued incidence matrices describing decompositions of the form: 
$$\design[2]{v,\frac{v-1}{2},\frac{v-3}{4}} \;\oplus\; \design[2]{v,\frac{v-1}{2},\frac{v-3}{4}} \;\oplus\; \design[2]{v,1,0}.$$

Specifically, for $v=31$,  we have the example 
and we ask, if, particularly for $v=b=31$, there exist other decompositions of this kind.

Purely arithmetically, we may think of 
$$\design[2]{31,15,7} \;\oplus\; \design[2]{31,10,3} \;\oplus\; \design[2]{31,6,1},$$ however, so far, we have not been able to provide an example of a $3$-valued incidence matrix giving rise to this decomposition.

This paper is devoted to developing a formalisation of the general idea underlying these thoughts. To the best of our knowledge, there has not been any approach to this phenomenon in the past, and we hope to contribute to a new chapter in the theory of combinatorial designs and possibly also that of designs over ${\rm GF}(q)$.

\section{Main Definition and Necessary Conditions}

Let us recapitulate the main notion introduced in the previous section. We can take the 
all-one matrix of dimensions $v \times b$ and try to write it as a sum of incidence matrices of 
$t$-designs: 
$$
J = M_1 + M_2 + \cdots + M_c ,
$$
where $M_j$'s are incident matrices of designs ${\mathcal B}_j$. 
Since the complement of a design is a design again, this equation is equivalent to the
next one: 
$$
\overline{M}_i = M_1 + M_2 + \cdots + M_{i-1} + M_{i+1} + \cdots + M_c .
$$
If we look at the last expression structurally, we see that we have decomposed every block
of the design $\overline{{\mathcal B}}_i$ into $c-1$ differently coloured subblocks in the way that every set of $b$ subblocks of the same colour constitutes a design. This fact justifies the way we have written down our constellations: 
\begin{equation} \label{defmosaic}
\design{v,v,b} = \design[t_1]{v,k_1,\lambda_1} \oplus \design[t_2]{v,k_2,\lambda_2}
\oplus \cdots \oplus \design[t_c]{v,k_c,\lambda_c},
\end{equation}
or equivalently
\begin{multline} \label{defmoscompl}
\design[t_i]{v,\overline{k}_i,\overline{\lambda}_i} = \design[t_1]{v,k_1,\lambda_1} \oplus \design[t_2]{v,k_2,\lambda_2}
\oplus \cdots \oplus \design[t_{i-1}]{v,k_{i-1},\lambda_{i-1}} \\
\oplus \design[t_{i+1}]{v,k_{i+1},\lambda_{i+1}} \oplus \cdots \oplus \design[t_c]{v,k_c,\lambda_c} .
\end{multline}
We would like to point out that it is more natural and elegant to look at this decomposition on the level of incidence matrices. 

\begin{definition}
	Let $c$ be a positive integer and let $\mathcal B_i$ be designs with 
	parameters  $\design[t_i]{v,k_i,\lambda_i}$, $i=1,\dots,c$ with the same 
	number of points $v$ and blocks $b$. A $c$-mosaic of designs 
	${\mathcal B}_1, {\mathcal B}_2, \dots,{\mathcal B}_c$ is a $(v \times b)$ matrix 
	$M = \left[m_{pq} \right]$, $m_{pq} \in \{l_1, l_2,\dots,l_c\}$ for 
	which holds that matrices $M_i$ defined as 
	$$
	\left[ M_i\right]_{pq} = \left\lbrace 
	\begin{array}{cl}
	1, & m_{pq}=l_i \\
	0, & otherwise \\
	\end{array}
	\right.  
	$$
	are incidence matrices of designs ${\mathcal B}_i$.
\end{definition}
Since there are a lot of parameters to be mentioned, we find it most convenient to denote a $c$-mosaic of designs with given parameters in the way as written in equations (\ref{defmosaic}) and (\ref{defmoscompl}). 
\begin{remark}\label{necessaryconditions}
The prerequisites of this definition lead, presented more explicitely, to the following facts: 
\begin{enumerate}
	\item 
	The integrality conditions for each parameter set $t_i$-$(v,k_i,\lambda_i)$ are satisfied.
	\item
	$
	b = \lambda_1 \frac{{v \choose t_1}}{{k_1 \choose t_1}} = \lambda_2 \frac{{v \choose t_2}}{{k_2 \choose t_2}} = \cdots = \lambda_c \frac{{v \choose t_c}}{{k_c \choose t_c}}
	$
\end{enumerate}
\end{remark}
\begin{remark}\label{completemosaiccharacteristics}
	If $M$ is a $c$-mosaic of designs with parameters given above, then it holds: 
	\begin{enumerate}
		\item
		$k_1 + k_2 + \cdots + k_c = v$,
		\item
		$r_1 + r_2 + \cdots + r_c = b$.
	\end{enumerate}
\end{remark}
Conditions listed in the two remarks above form necessary conditions for the existence of a $c$-mosaic of designs. We are not aware of any further necessary conditions to be satisfied. 

The following equations follow immediately from the existence of a $c$-mosaic: 
\begin{eqnarray}
M_1 + M_2 + \cdots + M_c & = & J_{v,b} \\
l_1 M_1 + l_2 M_2 + \cdots + l_c M_c & = & M
\end{eqnarray} 

\begin{remark}
	Given $c$ disjoint designs, such that $k_1+ \cdots + k_c < v$ the remaining entries in the $v \times r$ incidence matrix form a (multiset) $1$-design. We sometimes refer to such mosaics as \emph{partial} mosaics, discarding the final trivial design.
\end{remark}
%
\begin{example}
\begin{enumerate}[\bf(a)]
\item
The matrix of order $v$
$$
M= \left[ 
\begin{array}{cccccc}
1 & 2 & 3 & \cdots & v-1 & v \\
v & 1 & 2 & \cdots & v-2 & v-1 \\
v-1 & v & 1 & \cdots & v-3 & v-2 \\
\vdots & \vdots & \vdots & \vdots & \ddots & \vdots \\
2 & 3 & 4 & \cdots & v & 1 
\end{array}
\right]
$$
is a trivial $v$-mosaic of designs with parameters $\design[2]{v,1,0}$. 
\item
Incidence matrices of a $t$-design and its complement, coloured in any two different colours, define a $2$-mosaic. 
\item
The previously mentioned decomposition consisting of two Hadamard $2$-designs on $v$ points and a the trivial $\design[2]{v,1,0}$ design define a $3$-mosaic.
\end{enumerate}
\end{example} 

\section{Further Constructions of Mosaics}

The next obvious question for our investigation was to find non-trivial mosaics of designs. In particular, we have concentrated on $2$-designs, and how they can be composed together in such a mosaic. Before presenting a general theorem, we will show, how we tackled this problem given the example of affine planes.

\begin{example}
\begin{enumerate}[{\bf (a)}]
\item The incidence matrix of an affine plane of order $3$, i.e.~the $2$-design with parameters $(9,3,1)$, is a $(9 \times 12)$-matrix. Clearly, $3$ copies of such planes satisfy the necessary conditions and we come up with the following $3$-mosaic.
$$\left[\begin{array}{c|ccc||c|ccc||c|ccc}
1 & 1 & 2 & 3 & 3 & 3 & 1 & 2 & 2 & 2 & 3 & 1\\
1 & 3 & 1 & 2 & 3 & 2 & 3 & 1 & 2 & 1 & 2 & 3\\
1 & 2 & 3 & 1 & 3 & 1 & 2 & 3 & 2 & 3 & 1 & 2\\
\hline\hline
2 & 1 & 2 & 3 & 1 & 2 & 3 & 1 & 3 & 3 & 1 & 2\\
2 & 3 & 1 & 2 & 1 & 1 & 2 & 3 & 3 & 2 & 3 & 1\\
2 & 2 & 3 & 1 & 1 & 3 & 1 & 2 & 3 & 1 & 2 & 3\\
\hline\hline
3 & 1 & 2 & 3 & 2 & 1 & 2 & 3 & 1 & 1 & 2 & 3\\
3 & 3 & 1 & 2 & 2 & 3 & 1 & 2 & 1 & 3 & 1 & 2\\
3 & 2 & 3 & 1 & 2 & 2 & 3 & 1 & 1 & 2 & 3 & 1\\
\end{array}\right]$$
In our notation, this is a $3$-mosaic given by
$$
\design[2]{9,9,12}=\design[2]{9,3,1} \oplus \design[2]{9,3,1} \oplus \design[2]{9,3,1} .
$$
\item Being able to follow the same pattern, here is a $(16 \times 20)$-matrix representing a $4$-mosaic of $4$ copies of the affine plane of order $4$.
$$
\left[
\begin{array}{c|cccc||c|cccc||c|cccc||c|cccc}
0 & \bf 0 & \bf 1 & \bf 2 & \bf 3 & 1 & \bf 0 & 1 & 2 & 3 & 2 & \bf 0 & 1 & 2 & 3 & 3 & \bf 0 & 1 & 2 & 3 \\
0 & \bf 1 & \bf 0 & \bf 3 & \bf 2 & 1 & 1 & 0 & 3 & 2 & 2 & 1 & 0 & 3 & 2 & 3 & 1 & 0 & 3 & 2 \\
0 & \bf 2 & \bf 3 & \bf 0 & \bf 1 & 1 & 2 & 3 & 0 & 1 & 2 & 2 & 3 & 0 & 1 & 3 & 2 & 3 & 0 & 1 \\
0 & \bf 3 & \bf 2 & \bf 1 & \bf 0 & 1 & 3 & 2 & 1 & 0 & 2 & 3 & 2 & 1 & 0 & 3 & 3 & 2 & 1 & 0 \\ \hline\hline
1 & \bf 0 & 1 & 2 & 3 & 0 & \bf 1 & 0 & 3 & 2 & 3 & \bf 2 & 3 & 0 & 1 & 2 & \bf 3 & 2 & 1 & 0 \\
1 & 1 & 0 & 3 & 2 & 0 & 0 & 1 & 2 & 3 & 3 & 3 & 2 & 1 & 0 & 2 & 2 & 3 & 0 & 1 \\
1 & 2 & 3 & 0 & 1 & 0 & 3 & 2 & 1 & 0 & 3 & 0 & 1 & 2 & 3 & 2 & 1 & 0 & 3 & 2 \\
1 & 3 & 2 & 1 & 0 & 0 & 2 & 3 & 0 & 1 & 3 & 1 & 0 & 3 & 2 & 2 & 0 & 1 & 2 & 3 \\ \hline\hline
2 & \bf 0 & 1 & 2 & 3 & 3 & \bf 2 & 3 & 0 & 1 & 0 & \bf 3 & 2 & 1 & 0 & 1 & \bf 1 & 0 & 3 & 2 \\
2 & 1 & 0 & 3 & 2 & 3 & 3 & 2 & 1 & 0 & 0 & 2 & 3 & 0 & 1 & 1 & 0 & 1 & 2 & 3 \\
2 & 2 & 3 & 0 & 1 & 3 & 0 & 1 & 2 & 3 & 0 & 1 & 0 & 3 & 2 & 1 & 3 & 2 & 1 & 0 \\
2 & 3 & 2 & 1 & 0 & 3 & 1 & 0 & 3 & 2 & 0 & 0 & 1 & 2 & 3 & 1 & 2 & 3 & 0 & 1 \\ \hline\hline
3 & \bf 0 & 1 & 2 & 3 & 2 & \bf 3 & 2 & 1 & 0 & 1 & \bf 1 & 0 & 3 & 2 & 0 & \bf 2 & 3 & 0 & 1 \\
3 & 1 & 0 & 3 & 2 & 2 & 2 & 3 & 0 & 1 & 1 & 0 & 1 & 2 & 3 & 0 & 3 & 2 & 1 & 0 \\
3 & 2 & 3 & 0 & 1 & 2 & 1 & 0 & 3 & 2 & 1 & 3 & 2 & 1 & 0 & 0 & 0 & 1 & 2 & 3 \\
3 & 3 & 2 & 1 & 0 & 2 & 0 & 1 & 2 & 3 & 1 & 2 & 3 & 0 & 1 & 0 & 1 & 0 & 3 & 2 \\
\end{array}
\right]
$$
\end{enumerate}
\end{example}

The bold entries in part {\bf (b)} of the previous example shows how  ${\rm GF}(4)$ 
can be used in order to describe the decomposition.
Let 
$$
A:=\left[
\begin{tabular}{cccc}
0 & 1 & 2 & 3 \\
1 & 0 & 3 & 2 \\
2 & 3 & 0 & 1 \\
3 & 2 & 1 & 0 \\
\end{tabular}
\right]\;\;
\mbox{and} \;\;
M:=\left[
\begin{tabular}{cccc}
0 & 0 & 0 & 0 \\
0 & 1 & 2 & 3 \\
0 & 2 & 3 & 1 \\
0 & 3 & 1 & 2 \\
\end{tabular}
\right]
$$
denote the addition and multiplication table of ${\rm GF}(4)$. Then, up to column rearrangements, the above matrix can be written as: {\footnotesize
$$
\left[
\begin{array}{cccc|cccc}
m_{0,0} + A & m_{0,1} + A & m_{0,2} + A & m_{0,3} + A & a_{0,0} J_{4,1}  &  a_{0,1} J_{4,1} &  a_{0,2}J_{4,1} & a_{0,3} J_{4,1}  \\[1mm]
m_{1,0} + A & m_{1,1} + A & m_{1,2} + A & m_{1,3} + A & a_{1,0} J_{4,1} & a_{1,1}  J_{4,1}  &  a_{1,2}J_{4,1} & a_{1,3} J_{4,1} \\[1mm]
m_{2,0} + A & m_{2,1} + A & m_{2,2} + A & m_{2,3} + A & a_{2,0}J_{4,1}  & a_{2,1} J_{4,1}  & a_{2,2} J_{4,1}  & a_{2,3} J_{4,1} \\[1mm]
m_{3,0} + A & m_{3,1} + A & m_{3,2} + A & m_{3,3} + A & a_{3,0} J_{4,1} & a_{3,1} J_{4,1}  & a_{3,2}J_{4,1}  &  a_{3,3} J_{4,1}
\end{array}
\right]
$$}

where $J_{4,1}$ denotes the all-$1$-column of length $4$. In fact, it is clear that this mosaic allows a "chained" tactical decomposition (see \cite{tact}) of all copies of affine planes involved in this construction.

\begin{theorem}
Let $F$ be the field with $q$ elements. Then there is a $q$-mosaic of affine planes of order $q$:
$$
\design[2]{q^2,q^2,q^2+q} = \design[2]{q^2,q,1} \oplus \cdots \oplus \design[2]{q^2,q,1} .
$$
\end{theorem}
\begin{proof} Let $A$ and $M$ denote the operation tables (matrices) of the additive and multiplicative groups of $F$ respectively.  Then the desired $q$-mosaic is described and represented as a matrix of dimensions $q^2 \times (q^2+q)$ by the simple formula: 
$$
\left[ M \otimes J_{q,q} + J_{q,q}\otimes A \mid A \otimes J_{q,1}\right].
$$
Here $\otimes$ denotes the Kronecker product of matrices, and $J_{q,q}$ is the all-$1$ matrix of size $q\times q$. Now it is an easy task to verify that each element of $F$ (representing here a colour) appears in each row $q+1$ times and in each column $q$ times and that it appears in two different rows in the same columns just once at the same time.
\end{proof}
\subsection{Mosaics of Resolvable Designs}

We present a general construction for creating $\frac{v}{k}$-mosaics out of identical copies of a resolvable $\design{v,k,\lambda}$ design, i.e. 
\[ \design{v,v,b}=\bigoplus_{i=1}^{\frac{v}{k}} \design{v,k,\lambda}. \]

\begin{definition}
 A design $\mathcal{B}$ on a set $X$ is called resolvable if there exists a partition of the set of blocks $\mathcal{B}$ into so called parallel classes, such that every parallel class itself is a partition of the set of points $X$.
\end{definition}

We note that for a resolvable $\design{v, k, \lambda}$ design the number of parallel classes is $r:=\lambda_1$ and each class contains $\frac{v}{k}$ blocks.\\
Among the many examples of resolvable designs, affine planes and Kirkman triple systems are the most prominent ones.

\begin{theorem}
Let $D$ be the incidence matrix of a resolvable $\design{v,k,\lambda}$ design, where the columns have been arranged by parallel classes. Let $L$ be a latin square of order $\frac{v}{k}$ with entries $l_1, \dots, l_\frac{v}{k}$. Then $M:= D (I_r \otimes L)$ is a $\frac{v}{k}$-mosaic.
\end{theorem}
\begin{proof}
We see that every $v \times \frac{v}{k}$ submatrix containing one full parallel class is multiplied by a copy of the latin square $L$. Since every value $l_1, \dots, l_\frac{v}{k}$ appears exactly once in every column of $L$, every column of $M$ contains every block of a parallel class, each block multiplied by a different $l_i$. Since these columns form a parallel class they completely partition the set of points while having disjoint support, therefore completely covering the column without overlapping each other.\\
Furthermore every $l_i$ appears exactly once in every row of $(I_r \otimes L)$, therefore every block of the original design is contained in every design of the mosaic, which shows that each design is an exact copy of the original resolvable design.
\end{proof}

\begin{example}
A $\design[2]{15,3,1}$ resolvable design arranged by parallel classes.

\resizebox{\textwidth}{!}{
$ \arraycolsep=1.4pt 
D:=\left[ \begin{array}{ccccc|ccccc|ccccc|ccccc|ccccc|ccccc|ccccc}
1& 0& 0& 0& 0& 1& 0& 0& 0& 0& 0& 0& 0& 0& 1& 0& 0& 1& 0& 0& 0& 0& 0& 1& 0& 0& 0& 0& 0& 1& 0& 0& 0& 1& 0\\0& 1& 0& 0& 0& 1& 0& 0& 0& 0& 1& 0& 0& 0& 0& 0& 0& 0& 0& 1& 0& 0& 0& 0& 1& 0& 0& 1& 0& 0& 0& 0& 1& 0& 0\\0& 0& 1& 0& 0& 0& 1& 0& 0& 0& 1& 0& 0& 0& 0& 0& 0& 0& 1& 0& 1& 0& 0& 0& 0& 0& 0& 0& 1& 0& 0& 0& 0& 1& 0\\0& 0& 0& 1& 0& 0& 1& 0& 0& 0& 0& 1& 0& 0& 0& 0& 0& 0& 0& 1& 0& 1& 0& 0& 0& 0& 0& 0& 0& 1& 1& 0& 0& 0& 0\\0& 0& 0& 0& 1& 1& 0& 0& 0& 0& 0& 1& 0& 0& 0& 1& 0& 0& 0& 0& 1& 0& 0& 0& 0& 1& 0& 0& 0& 0& 0& 1& 0& 0& 0\\1& 0& 0& 0& 0& 0& 0& 0& 0& 1& 1& 0& 0& 0& 0& 1& 0& 0& 0& 0& 0& 1& 0& 0& 0& 0& 1& 0& 0& 0& 0& 0& 0& 0& 1\\0& 1& 0& 0& 0& 0& 1& 0& 0& 0& 0& 0& 0& 0& 1& 0& 1& 0& 0& 0& 0& 0& 1& 0& 0& 1& 0& 0& 0& 0& 0& 0& 0& 0& 1\\0& 0& 1& 0& 0& 0& 0& 1& 0& 0& 0& 1& 0& 0& 0& 0& 1& 0& 0& 0& 0& 0& 0& 1& 0& 0& 1& 0& 0& 0& 0& 0& 1& 0& 0\\0& 0& 0& 1& 0& 0& 0& 1& 0& 0& 0& 0& 1& 0& 0& 1& 0& 0& 0& 0& 0& 0& 1& 0& 0& 0& 0& 1& 0& 0& 0& 0& 0& 1& 0\\0& 0& 0& 0& 1& 0& 0& 0& 1& 0& 0& 0& 1& 0& 0& 0& 0& 0& 0& 1& 0& 0& 0& 1& 0& 0& 0& 0& 1& 0& 0& 0& 0& 0& 1\\1& 0& 0& 0& 0& 0& 0& 0& 1& 0& 0& 0& 0& 1& 0& 0& 1& 0& 0& 0& 1& 0& 0& 0& 0& 0& 0& 1& 0& 0& 1& 0& 0& 0& 0\\0& 1& 0& 0& 0& 0& 0& 1& 0& 0& 0& 0& 0& 1& 0& 0& 0& 1& 0& 0& 0& 1& 0& 0& 0& 0& 0& 0& 1& 0& 0& 1& 0& 0& 0\\0& 0& 1& 0& 0& 0& 0& 0& 0& 1& 0& 0& 1& 0& 0& 0& 0& 1& 0& 0& 0& 0& 0& 0& 1& 1& 0& 0& 0& 0& 1& 0& 0& 0& 0\\0& 0& 0& 1& 0& 0& 0& 0& 1& 0& 0& 0& 0& 0& 1& 0& 0& 0& 1& 0& 0& 0& 0& 0& 1& 0& 1& 0& 0& 0& 0& 1& 0& 0& 0\\0& 0& 0& 0& 1& 0& 0& 0& 0& 1& 0& 0& 0& 1& 0& 0& 0& 0& 1& 0& 0& 0& 1& 0& 0& 0& 0& 0& 0& 1& 0& 0& 1& 0& 0
\end{array}\right]$}\\

We choose a simple cyclic shift for the latin square
$$ L:=\begin{bmatrix}
1 &5 &4 &3 &2\\
2 &1 &5 &4 &3\\
3 &2 &1 &5 &4\\
4 &3 &2 &1 &5\\
5 &4 &3 &2 &1\\
\end{bmatrix}$$
and calculate the incidence matrix of a mosaic as $M = D (I_7 \otimes L)$.\\

\resizebox{\textwidth}{!}{
$ \arraycolsep=1.4pt 
M:=\left[ \begin{array}{ccccc|ccccc|ccccc|ccccc|ccccc|ccccc|ccccc}
1& 5& 4& 3& 2& 1& 5& 4& 3& 2& 5& 4& 3& 2& 1& 3& 2& 1& 5& 4& 4& 3& 2& 1& 5& 5& 4& 3& 2& 1& 4& 3& 2& 1& 5\\2& 1& 5& 4& 3& 1& 5& 4& 3& 2& 1& 5& 4& 3& 2& 5& 4& 3& 2& 1& 5& 4& 3& 2& 1& 3& 2& 1& 5& 4& 3& 2& 1& 5& 4\\3& 2& 1& 5& 4& 2& 1& 5& 4& 3& 1& 5& 4& 3& 2& 4& 3& 2& 1& 5& 1& 5& 4& 3& 2& 4& 3& 2& 1& 5& 4& 3& 2& 1& 5\\4& 3& 2& 1& 5& 2& 1& 5& 4& 3& 2& 1& 5& 4& 3& 5& 4& 3& 2& 1& 2& 1& 5& 4& 3& 5& 4& 3& 2& 1& 1& 5& 4& 3& 2\\5& 4& 3& 2& 1& 1& 5& 4& 3& 2& 2& 1& 5& 4& 3& 1& 5& 4& 3& 2& 1& 5& 4& 3& 2& 1& 5& 4& 3& 2& 2& 1& 5& 4& 3\\1& 5& 4& 3& 2& 5& 4& 3& 2& 1& 1& 5& 4& 3& 2& 1& 5& 4& 3& 2& 2& 1& 5& 4& 3& 2& 1& 5& 4& 3& 5& 4& 3& 2& 1\\2& 1& 5& 4& 3& 2& 1& 5& 4& 3& 5& 4& 3& 2& 1& 2& 1& 5& 4& 3& 3& 2& 1& 5& 4& 1& 5& 4& 3& 2& 5& 4& 3& 2& 1\\3& 2& 1& 5& 4& 3& 2& 1& 5& 4& 2& 1& 5& 4& 3& 2& 1& 5& 4& 3& 4& 3& 2& 1& 5& 2& 1& 5& 4& 3& 3& 2& 1& 5& 4\\4& 3& 2& 1& 5& 3& 2& 1& 5& 4& 3& 2& 1& 5& 4& 1& 5& 4& 3& 2& 3& 2& 1& 5& 4& 3& 2& 1& 5& 4& 4& 3& 2& 1& 5\\5& 4& 3& 2& 1& 4& 3& 2& 1& 5& 3& 2& 1& 5& 4& 5& 4& 3& 2& 1& 4& 3& 2& 1& 5& 4& 3& 2& 1& 5& 5& 4& 3& 2& 1\\1& 5& 4& 3& 2& 4& 3& 2& 1& 5& 4& 3& 2& 1& 5& 2& 1& 5& 4& 3& 1& 5& 4& 3& 2& 3& 2& 1& 5& 4& 1& 5& 4& 3& 2\\2& 1& 5& 4& 3& 3& 2& 1& 5& 4& 4& 3& 2& 1& 5& 3& 2& 1& 5& 4& 2& 1& 5& 4& 3& 4& 3& 2& 1& 5& 2& 1& 5& 4& 3\\3& 2& 1& 5& 4& 5& 4& 3& 2& 1& 3& 2& 1& 5& 4& 3& 2& 1& 5& 4& 5& 4& 3& 2& 1& 1& 5& 4& 3& 2& 1& 5& 4& 3& 2\\4& 3& 2& 1& 5& 4& 3& 2& 1& 5& 5& 4& 3& 2& 1& 4& 3& 2& 1& 5& 5& 4& 3& 2& 1& 2& 1& 5& 4& 3& 2& 1& 5& 4& 3\\5& 4& 3& 2& 1& 5& 4& 3& 2& 1& 4& 3& 2& 1& 5& 4& 3& 2& 1& 5& 3& 2& 1& 5& 4& 5& 4& 3& 2& 1& 3& 2& 1& 5& 4
\end{array}\right]$}
\end{example}
The earlier examples on affine planes can be constructed in this fashion as well.

\section{$q$-analogues of Mosaics}

To define mosaics of $q$-designs a $q$-analogue of the disjoint property of blocks is necessary. It seems most natural to consider a set of blocks to be 'disjoint' if they are linearly independent, i.e.
\[ \dim\left( \sum_i B_i \right)=\sum_i\dim(B_i) \]
or, in other words, their sum is direct. 

\begin{definition}
	Let $c$ be a positive integer and let $\mathcal{B}_i$ be $q$-designs with parameters $\design[t_i]{v,k_i,\lambda_i;q}$, $1 \leq i \leq c$, with the same number of blocks $b$. We denote by $B_i^j$, $1 \leq j \leq b$, the different blocks of design $\mathcal{B}_i$. We say the designs $\mathcal{B}_i$ form a $c$-mosaic if 
		\[\sum_i k_i= \dim\left( \sum_i B_i^j \right)= v, \;\; \text{ for all } \; j.\]
\end{definition}

\begin{remark}
	If we have designs such that $\sum_i k_i= \dim\left( \sum_i B_i^j \right)$ for all $j,$ we speak of a \emph{partial} mosaic of $q$-designs.
\end{remark}
\begin{remark}
	As in remark \ref{necessaryconditions} the following necessary conditions for a (partial) mosaic of $q$-designs follow from the existence of the involved $q$-designs:
	\begin{enumerate}
		\item The integrality conditions for each parameter set $\design[t_i]{v,k_i,\lambda_i;q}$ are satisfied
		\item 	$
		b = \lambda_1 \frac{{v \brack t_1}_q}{{k_1 \brack t_1}_q} = \lambda_2 \frac{{v \brack t_2}_q}{{k_2 \brack t_2}_q} = \cdots = \lambda_c \frac{{v \brack t_c}_q}{{k_c \brack t_c}_q}
		$
	\end{enumerate}
\end{remark}

\begin{remark}
	Similar to remark \ref{completemosaiccharacteristics} for a $c$-mosaic of $q$-designs as described above the following equations holds:
	\begin{enumerate}
		\item $k_1+k_2+ \cdots +k_c=v$
		\item $b (q^{k_i}-1) = (q^v-1) r_i$.
	\end{enumerate}
	We can put these equations together and see that
	\[q^v=\prod_i \left(\frac{q^v-1}{b}r_i+1\right)\]
\end{remark}

\begin{example}
It is well known that the Fano plane not only describes a $\design[2]{7,3,1}$ design but can also be seen as a $\design[2]{3,2,1;2}$ design over $\mathbb{F}_2^3$.
We can add a trivial $\design[1]{3,1,1;2}$ design to form a $2$-mosaic. See \ref{ex1} and consider only the colors $1$ and $2$.
\end{example}

\begin{example}
		In \cite{Braun} the authors construct the first known $\design[2]{13,3,1;2}$ designs. They use the well known isomorphisms $\ff_2^{13}\cong\ff_{2^{13}}$ and $\ff_{2^{13}}^* \cong \Z_{2^{13}-1}$  where the latter is defined through a generator $\alpha \in \ff_{2^{13}}^*$ with minimal polynomial $x^{13}+x^{12}+x^{10}+x^{9}+1$. The design is described as the union of $15$ orbits under the action of the group $A_\alpha:= Gal(\ff_{2^{13}} / \ff_2) \ltimes C_\alpha$, where $C_\alpha$ is the group generated by $M_\alpha$ the multiplication with $\alpha$. The following list contains one representative from each orbit, described by the exponents of the non-zero elements in the vector space.		
	\begin{align*}
	&V_1:=[0, 1, 1249, 5040, 7258, 7978, 8105],
	&&V_2:=[0, 7, 1857, 6681, 7259, 7381, 7908],\\
	&V_3:=[0, 9, 1144, 1945, 6771, 7714, 8102], 
	&&V_4:=[0, 11, 209, 1941, 2926, 3565, 6579],\\
	&V_5:=[0, 12, 2181, 2519, 3696, 6673, 6965],
	&&V_6:=[0, 13, 4821, 5178, 7823, 8052, 8110],\\
	&V_7:=[0, 17, 291, 1199, 5132, 6266, 8057], 
	&&V_8:=[0, 20, 1075, 3939, 3996, 4776, 7313],\\
	&V_9:=[0, 21, 2900, 4226, 4915, 6087, 8008],
	&&V_{10}:=[0, 27, 1190, 3572, 4989, 5199, 6710],\\
	&V_{11}:=[0, 30, 141, 682, 2024, 6256, 6406],
	&&V_{12}:=[0, 31, 814, 1161, 1243, 4434, 6254],\\
	&V_{13}:=[0, 37, 258, 2093, 4703, 5396, 6469], 
	&&V_{14}:=[0, 115, 949, 1272, 1580, 4539, 4873],\\
	&V_{15}:=[0, 119, 490, 5941, 6670, 6812, 7312]
\end{align*}
It is possible to arrange four copies of this design into a partial $4$-mosaic. We found that the first vector space and its images under multiplication with $\alpha^2, \alpha^4,$ and $\alpha^8$ form a direct sum.
\[ V_1 \oplus M_\alpha^2 V_1 \oplus M_\alpha^4 V_1 \oplus M_\alpha^8 V_1. \]
Applying the group action to these four summands gives us the first $106483$ blocks of our mosaic. Similarly computer search finds that the following are direct sums
\begin{align*}
& V_1 \oplus M_\alpha^2 V_1 \oplus M_\alpha^4 V_1 \oplus M_\alpha^8 V_1, &&V_2 \oplus M_\alpha^1 V_2 \oplus M_\alpha^4 V_2 \oplus M_\alpha^{12} V_2 \\
& V_3 \oplus M_\alpha^1 V_3 \oplus M_\alpha^2 V_3 \oplus M_\alpha^3 V_3, &&V_4 \oplus M_\alpha^1 V_4 \oplus M_\alpha^2 V_4 \oplus M_\alpha^3 V_4 \\
& V_5 \oplus M_\alpha^1 V_5 \oplus M_\alpha^2 V_5 \oplus M_\alpha^3 V_5, &&V_6 \oplus M_\alpha^1 V_6 \oplus M_\alpha^2 V_6 \oplus M_\alpha^3 V_6 \\
&V_7 \oplus M_\alpha^1 V_7 \oplus M_\alpha^2 V_7 \oplus M_\alpha^3 V_7, &&V_8 \oplus M_\alpha^1 V_8 \oplus M_\alpha^2 V_8 \oplus M_\alpha^3 V_8 \\
&V_9 \oplus M_\alpha^1 V_9 \oplus M_\alpha^2 V_9 \oplus M_\alpha^3 V_9, &&V_{10} \oplus M_\alpha^1 V_{10} \oplus M_\alpha^2 V_{10} \oplus M_\alpha^3 V_{10} \\
&V_{11} \oplus M_\alpha^1 V_{11} \oplus M_\alpha^2 V_{11} \oplus M_\alpha^3 V_{11}, &&V_{12} \oplus M_\alpha^1 V_{12} \oplus M_\alpha^2 V_{12} \oplus M_\alpha^4 V_{12} \\
&V_{13} \oplus M_\alpha^1 V_{13} \oplus M_\alpha^2 V_{13} \oplus M_\alpha^4 V_{13}, &&V_{14} \oplus M_\alpha^1 V_{14} \oplus M_\alpha^2 V_{14} \oplus M_\alpha^3 V_{14} \\
&V_{15} \oplus M_\alpha^1 V_{15} \oplus M_\alpha^2 V_{15} \oplus M_\alpha^4 V_{15}\\
\end{align*}
We therefore have a partial mosaic of four copies of the $\design[2]{13,3,1;2}$ design. We can complete the mosaic with a trivial multi $\design[1]{13,1,195;2}$ design containing $195$ copies of each one dimensional vector space. This is achieved by completing each of the above direct sums with a fifth summand of dimension one and having the group $C_\alpha$ act on it. The group acts transitively on the one dimensional vector spaces with stabilizer of size $13$. Therefore every one dimensional vector space is repeated $13$ times within each of the $15$ orbits and therefore $195$ times in total.
\end{example}

\section{Applications of Mosaics}
Combinatorial design theory has its roots in the design of experiments \cite{fisher}, arranging test subjects in suitable test groups, or arranging players in a tournament. Assume $v$ players and a game that needs exactly $k$ participants, then using a $\design{v,k,\lambda}$ design ensures that when playing $b$ games every subset of $t$ subjects play together exactly $\lambda$ times. Resolvable designs have been very useful, since they allow for a parallelization of test runs that are in the same parallel class.\\
With games and experiments it is quite often the case that a parallelization in this manner is not possible, either because only one gameboard is at hand or it is required for the test to be supervised by the same person to ensure consistency. In this case a mosaic can be used to parallelize several different games/tests, each with its own required number of players $k_i$, on the same set of subjects. Assume a mosaic of the form
\[ \design[t_1]{v,k_1, \lambda_1} \oplus \cdots \oplus \design[t_c]{v,k_c, \lambda_c}  \]
then every column of the mosaic partitions the different players onto the different games, while preserving the properties of the designs used in the mosaic.\\

A different application of mosaics lies in media access control. Mosaics of $2$-designs generated from disjoint difference sets have been shown to provide significant improvements in rendevouz probability when channel hopping is used while avoiding the need for any centralized organization \cite{MAC}.

Furthermore applications of designs in distributed storage systems (DSS) have recently been discussed in several papers under the notion Fractional Repetition Codes (FRC). For more information, see~\cite{Silberstein} and its bibliography. It is conceivable that mosaics will provide novel features to these systems, as they will contribute to parallelisation and thus increased efficiency.

\section*{Conclusion}

In this paper we have introduced the notion of a {\em mosaic\/} to the theory of combinatorial designs. Mosaics may be thought of as tilings of the ambient space of a design with disjoint copies of this design or otherwise with disjoint blocks of designs with different parameters. We were able to generalize our ideas to $q$-analogs of $t$-designs and show the first results in this context.

Beside the rather immediate constructions of infinite families of mosaics derived from resolvable designs, we have a few open problems: How can mosaics be derived from non-resolvable designs? How can the concrete open case after example~1.1 (before the second section of this paper) be tackled?

\printbibliography

\end{document}